\documentclass[11pt,reqno]{amsart}

\usepackage[T1]{fontenc}
\usepackage{lmodern}
\usepackage{amsmath,amssymb,amsthm,mathtools,stmaryrd}
\usepackage{mathrsfs}
\usepackage{microtype}
\usepackage[margin=1.12in]{geometry}
\usepackage[colorlinks=true,linkcolor=blue,citecolor=blue,urlcolor=blue]{hyperref}
\usepackage{aliascnt}
\usepackage[nameinlink,noabbrev,capitalize]{cleveref}

\allowdisplaybreaks
\numberwithin{equation}{section}

\newtheorem{theorem}{Theorem}[section]
\newaliascnt{lemma}{theorem}
\newtheorem{lemma}[lemma]{Lemma}
\aliascntresetthe{lemma}
\newaliascnt{proposition}{theorem}
\newtheorem{proposition}[proposition]{Proposition}
\aliascntresetthe{proposition}
\newaliascnt{corollary}{theorem}
\newtheorem{corollary}[corollary]{Corollary}
\aliascntresetthe{corollary}
\theoremstyle{definition}
\newaliascnt{problem}{theorem}
\newtheorem{problem}[problem]{Problem}
\aliascntresetthe{problem}
\theoremstyle{remark}
\newaliascnt{remark}{theorem}
\newtheorem{remark}[remark]{Remark}
\aliascntresetthe{remark}

\crefname{theorem}{Theorem}{Theorems}
\Crefname{theorem}{Theorem}{Theorems}
\crefname{lemma}{Lemma}{Lemmas}
\Crefname{lemma}{Lemma}{Lemmas}
\crefname{proposition}{Proposition}{Propositions}
\Crefname{proposition}{Proposition}{Propositions}
\crefname{corollary}{Corollary}{Corollaries}
\Crefname{corollary}{Corollary}{Corollaries}
\crefname{problem}{Problem}{Problems}
\Crefname{problem}{Problem}{Problems}
\crefname{remark}{Remark}{Remarks}
\Crefname{remark}{Remark}{Remarks}

\newcommand{\Hyp}{\mathbb H}
\newcommand{\R}{\mathbb R}

\newcommand{\Mass}{\mathbf M}
\newcommand{\supp}{\operatorname{supp}}
\newcommand{\dist}{\operatorname{dist}}
\newcommand{\Proj}{\operatorname{Proj}}
\newcommand{\tr}{\operatorname{tr}}
\newcommand{\restr}{\mathbin{\llcorner}}
\newcommand{\cur}[1]{\left[\!\left[#1\right]\!\right]}
\newcommand{\eps}{\varepsilon}

\title[Mass bounds for confined area-minimizing minimal surfaces]
{Mass Bounds for Confined Area-Minimizing Minimal Surfaces}

\author{Xumin Jiang}
\address{School of Sciences, Great Bay University, Dongguan 523000, China}
\email{xjiang@gbu.edu.cn}
\author{Jiongduo Xie}
\address{School of Mathematical Sciences, Shanghai Jiao Tong University,
Shanghai 200240, China}
\email{jiongduoxie@outlook.com}


\begin{document}

\begin{abstract}
We establish codimension-independent mass bounds for geometrically confined
area-minimizing rectifiable currents.   In Euclidean
space, we combine the confined-volume doubling theorem of
Colding--Minicozzi with a current-theoretic squashing argument.  This gives an
affirmative answer to Lin's interior mass-bound problem for every algebraic
projection multiplicity $Q$: the interior mass is bounded by $C(n)Q$, without
an a priori mass bound at a larger scale. This result  yields a degree‑one Bernstein‑type rigidity result under sublinear confinement.

For the hyperbolic application, we make the curvature modification of the
fixed-scale argument needed in a thin tubular neighborhood of a totally
geodesic copy of $\Hyp^n$.  Combining this auxiliary
estimate with a localized squashing estimate removes the doubly exponential
local mass-growth condition from the boundary regularity  results in \cite{JiangXie}.
\end{abstract}

\maketitle

\section{Introduction}\label{sec:introduction}

\subsection{Background and motivation}

Let $n\ge2$, $k\ge1$, and $\mathfrak m=n+k$.  We use the same upper-half-space
notation as in \cite{JiangXie}:
\[
 \Hyp^{\mathfrak m}
 =\{(x',x^{\mathfrak m})\in\R^{\mathfrak m-1}\times\R_+\},
 \qquad
 g_{\Hyp}=(x^{\mathfrak m})^{-2}|dx|^2.
\]
The asymptotic Plateau problem asks for complete minimal submanifolds in
hyperbolic space with prescribed boundary at infinity.  Anderson established
existence for broad classes of boundary data and initiated the systematic
geometric study of the problem \cite{Anderson1982,Anderson1983}.  In the
hypersurface case, Hardt--Lin proved regularity at infinity for
area-minimizing hypersurfaces \cite{HardtLin}, and Lin subsequently studied
their asymptotic behavior and the minimal-graph Dirichlet problem
\cite{LinAsymptotic,LinDirichlet,LinErratum}.  Lin also treated existence for
higher-codimensional area-minimizing currents and flat chains modulo $p$, and
proved boundary regularity at infinity for the modulo-$2$ problem
\cite{LinAsymptotic}.  Boundary regularity for
minimal graphs has since been developed further in
\cite{HanShenWang,HanJiang,HanWang}; see \cite{CoskunuzerSurvey} for a survey
of the asymptotic Plateau problem.  Related variational existence and
regularity questions for constant-mean-curvature hypersurfaces were treated
by Tonegawa \cite{Tonegawa}.

Higher codimension is qualitatively different.  Singular and branched
area-minimizing currents occur even in Euclidean space; the foundational
regularity theory is due to Federer and Almgren
\cite{Federer,Almgren}, while the Lawson--Osserman examples illustrate the
special difficulties of higher-codimensional minimal graphs
\cite{LawsonOsserman}.  Boundary regularity for varifolds is governed by
Allard's interior and boundary theories \cite{AllardFirst,AllardBoundary};
hyperbolic boundary regularity for higher-codimensional currents is therefore
not a formal consequence of the hypersurface theory.

An elementary-looking mass question identified by Lin is particularly
relevant here.  We record it in current notation.  Let
$\mathbf p:\R^n\times\R^k\to\R^n$ be the orthogonal projection, let
$D_r=B_r^n(0)$, endowed with its standard orientation and identified with
$D_r\times\{0\}\subset\R^n\times\R^k$, and let
$\mathcal C_r=D_r\times\R^k$.

Lin posed the following problem in
\cite[Section~3, Problem~1]{LinAsymptotic}.

\begin{problem}[Lin's Problem 1]
\label{prob:Lin}
Let $T$ be an indecomposable, $n$-dimensional area-minimizing integral
current in $\R^{n+k}$ such that
\[
 \mathbf p_\#T=Q\cur{D_1},\qquad
 (\partial T)\restr\mathcal C_1=0,
\]
and the Hausdorff distance between $\supp T$ and $D_1$ is at
most $\delta\ll1$.  Does smallness of $\delta$ imply
\[
 \Mass(T\restr\mathcal C_{1/2})\le C(Q,n,k)?
\]
\end{problem}

Lin emphasized that the question was open even for $Q=1$ when $T$ is a
smooth graph over $D_1$ \cite[Section~3]{LinAsymptotic}.

\subsection{Statement of the results}

The first main result gives a positive answer, in fact with no need for the
indecomposability assumption and with constants independent of $k$.
indecomposability assumption.

\begin{theorem}
\label{thm:Lin-problem}
There are constants $\delta_L=\delta_L(n)>0$ and
$C_L=C_L(n)<\infty$ with the following property.  Let $T$ be an
$n$-dimensional area-minimizing integral current in $\R^{n+k}$ satisfying
\begin{equation}\label{eq:Lin-hypotheses}
 \mathbf p_\#T=Q\cur{D_1},
 \qquad
 (\partial T)\restr\mathcal C_1=0,
 \qquad
 d_{\rm H}(D_1,\supp T)\le\delta_L,
\end{equation}
where $Q\ge1$.  Then
\begin{equation}\label{eq:Lin-conclusion}
 \Mass(T\restr\mathcal C_{1/2})\le C_LQ.
\end{equation}
In particular, \cref{prob:Lin} has an affirmative answer for every $Q$.
\end{theorem}

The proof combines two estimates.  The first is the confined-volume theorem of
Colding--Minicozzi 
\cite{ColdingMinicozzi}.  The second is the Euclidean squashing estimate used
by Lin and in \cite{JiangXie}.  After
localizing the global identity $\mathbf p_\#T=Q\cur{D_1}$ to fixed interior
balls, the two inequalities close by absorption. 

The integer $Q$ must not be confused with the density of the associated
varifold $V_T$.  The former is the algebraic multiplicity of the projected
current: oppositely oriented sheets may cancel under $\mathbf p_\#$.  The
weight measure $\mu_{V_T}=\|T\|$, by contrast, counts all sheets positively.
The Colding--Minicozzi estimate is homogeneous in this total varifold mass
and by itself gives no bound in terms of $Q$.  In the squashing construction,
the cost of every sheet, including those that cancel after projection, is
retained in a term of the form $c\delta F'(r)$; only after this error is
combined with doubling and absorbed does the remaining term become
$C(n)Q$.

We state the Colding--Minicozzi result precisely in
\cref{thm:CM-doubling}. Some  applications of Theorem \ref{thm:Lin-problem}, including a degree‑one Bernstein‑type rigidity result under sublinear confinement, will be discussed in Section \ref{sec:app}.  For the application to \cite{JiangXie}, we also need
a corresponding fixed-scale estimate in hyperbolic space.

In \cite{JiangXie} we studied locally area-minimizing integer-rectifiable
$n$-currents in $\Hyp^{\mathfrak m}$ asymptotic to a closed oriented
$C^{1,\alpha}$ submanifold of the conformal boundary.  The local hypothesis
in \cite[Assumption~1.3]{JiangXie} has two logically distinct parts: the
multiplicity-one projection identity \cite[(1.9)]{JiangXie} and the local
mass-growth condition
\begin{equation}\label{eq:old-mass-growth}
 \Mass_{\Hyp^{\mathfrak m}}
 \bigl(T\restr B_{\Hyp^{\mathfrak m}}(P,2)\bigr)
 <\exp\!\left(c_0\bigl(x^{\mathfrak m}(P)\bigr)^{-\alpha}\right).
\end{equation}
The latter allows doubly exponential growth in the hyperbolic distance to a
fixed point.  In \cite[Section~3]{JiangXie}, it is combined with an
exponentially small squashing factor to obtain a uniform unit-ball mass
bound.  Our application shows that \eqref{eq:old-mass-growth} is
redundant.  Although the local mass estimate holds for every algebraic
projection multiplicity $Q$, the boundary regularity theorem retains
multiplicity one:
a mass bound does not by itself exclude higher-codimensional branching.

We work throughout in the local current-theoretic setting of
\cite[Assumption~1.3]{JiangXie}; no finite-mass or normal-current extension
to the Euclidean compactification is assumed.  If $V$ is a varifold,
$\mu_V=\|V\|$ denotes its weight measure.  Unless stated otherwise, balls,
distances, masses, and tubular neighborhoods are hyperbolic.

The next result is the fixed-scale hyperbolic estimate needed to close the
mass argument at infinity.  Its proof follows the mechanism of
\cite[Theorem~0.5]{ColdingMinicozzi}, with the curvature correction described
below.

\begin{theorem}
\label{thm:hyperbolic-comparison}
There exist constants
\[
 \eps_D=\eps_D(n)>0,
 \qquad
 C_D=C_D(n)<\infty,
\]
with the following property.  Let
$L\cong\Hyp^n\subset\Hyp^{\mathfrak m}$ be
totally geodesic, let $p\in L$, and let $V$ be a stationary integral
$n$-varifold in $B_{3/2}(p)$.  If
\begin{equation}\label{eq:confinement-main}
 \supp V\cap B_{3/2}(p)
 \subset\bigl\{x\in\Hyp^{\mathfrak m}:\dist_{\Hyp}(x,L)<\eps_D\bigr\},
\end{equation}
then
\begin{equation}\label{eq:mass-comparison-main}
 \mu_V\bigl(B_{5/4}(p)\bigr)
 \le C_D\,\mu_V\bigl(B_{1/4}(p)\bigr).
\end{equation}
The constants are independent of the codimension $k$.
\end{theorem}

The center of the ball need not lie exactly on the comparison plane.

\begin{corollary}\label{cor:shifted-comparison}
After decreasing $\eps_D(n)$, the following holds.  Let
$q\in\Hyp^{\mathfrak m}$, let $L\cong\Hyp^n$ be totally geodesic, and let
$V$ be an integral $n$-varifold in $B_2(q)$.  Suppose
\[
 \dist_{\Hyp}(q,L)<\eps_D,
\]
$V$ is stationary in $B_2(q)$, and
\[
 \supp V\cap B_2(q)
 \subset\{\dist_{\Hyp}(\cdot,L)<\eps_D\}.
\]
Then
\begin{equation}\label{eq:shifted-comparison}
 \mu_V(B_1(q))\le C_D\,\mu_V(B_{1/2}(q)).
\end{equation}
\end{corollary}

For the absolute mass estimate we use the height-preserving Euclidean
projection appearing in \cite{JiangXie}.  After an ambient isometry, split
$x'=(w,z)\in\R^{n-1}\times\R^k$ and write
\begin{equation}\label{eq:upper-half-space-splitting}
 \Hyp^{\mathfrak m}
 =\{(w,z,y):w\in\R^{n-1},\ z\in\R^k,\ y>0\},
 \qquad
 g_{\Hyp}=y^{-2}(dw^2+dz^2+dy^2),
\end{equation}
and let
\begin{equation}\label{eq:vertical-plane-projection}
 L=\{z=0\}\cong\Hyp^n,
 \qquad
 \Pi_L(w,z,y)=(w,0,y).
\end{equation}
Thus $\Pi_L$ is Euclidean orthogonal projection preserving the height
coordinate; it is not the intrinsic nearest-point projection.  For $p\in L$
write
\[
 D_s^L(p)=B_s^{\Hyp^{\mathfrak m}}(p)\cap L,
\]
and fix an orientation of $L$.

\begin{theorem}
\label{thm:local-mass}
There exist $\delta_*=\delta_*(n)>0$ and $C_*=C_*(n)<\infty$ with the
following property.  Let $T$ be an integer-rectifiable $n$-current with
locally finite mass in $\Hyp^{\mathfrak m}$, locally area minimizing, and let
$q\in\Hyp^{\mathfrak m}$.  Assume
\begin{equation}\label{eq:no-boundary-local}
 (\partial T)\restr B_2(q)=0.
\end{equation}
Suppose that, after an ambient isometry, $L$ and $\Pi_L$ have the form
\eqref{eq:vertical-plane-projection}.  Put $p=\Pi_L(q)$.  If, for some
integer $Q\ge1$ and $0<\delta\le\delta_*$,
\begin{align}
 \dist_{\Hyp}(q,L)&\le\delta,\label{eq:center-local}\\
 \supp T\cap B_2(q)&\subset
 \{\dist_{\Hyp}(\cdot,L)\le\delta\},\label{eq:tube-local}\\
 (\Pi_L)_\#\bigl(T\restr B_2(q)\bigr)\restr D_{3/2}^L(p)
 &=Q\,\cur{D_{3/2}^L(p)},\label{eq:projection-local}
\end{align}
then
\begin{equation}\label{eq:mass-local}
 \Mass_{\Hyp^{\mathfrak m}}\bigl(T\restr B_1(q)\bigr)\le C_*Q.
\end{equation}
The constants are independent of $k$, and no a priori upper bound for
$\Mass_{\Hyp^{\mathfrak m}}(T\restr B_2(q))$ is assumed.
\end{theorem}

The integer $Q$ in \eqref{eq:projection-local} is an algebraic projection
degree.  The theorem does not assume that the associated varifold $V_T$ has
density $Q$; additional sheets with opposite projected orientations are
allowed.

The multiplicity-one case is the form needed at infinity.  We use the
notation $G_R$, $B_R^+$, $H^+$, and $\Proj$ from
\cite[Section~1]{JiangXie}.  The next theorem retains the projection condition
\cite[(1.9)]{JiangXie} and omits the mass-growth condition
\cite[(1.10)]{JiangXie}.

\begin{theorem}
\label{thm:remove-growth}
Let $T$ be a locally area-minimizing integer-rectifiable $n$-current in
$\Hyp^{\mathfrak m}$, asymptotic to a closed oriented
$(n-1)$-dimensional $C^1$ submanifold
$\Gamma\subset\R^{\mathfrak m-1}\times\{0\}$.  Suppose that, for some
$R>0$ and $\alpha\in(0,1]$,
\[
 \Gamma\cap\overline{G_R}\in C^{1,\alpha},
\]
and that, for one $r\in(0,R)$,
\begin{equation}\label{eq:boundary-projection}
 \Proj\bigl(T\restr G_r\bigr)=\cur{B_r^+}.
\end{equation}
Then there exist $R_0\in(0,R)$ and $t_0>0$,
depending on the local $C^{1,\alpha}$ geometry of $\Gamma$, such that
\begin{equation}\label{eq:uniform-at-infinity}
 \sup_{\substack{P\in\supp T\cap G_{R_0}\\0<x^{\mathfrak m}(P)<t_0}}
 \Mass_{\Hyp^{\mathfrak m}}
 \bigl(T\restr B_{\Hyp^{\mathfrak m}}(P,1)\bigr)
 \le C(n).
\end{equation}
\end{theorem}

\begin{remark}
   Under the
corresponding additional regularity assumptions on $\Gamma$, the boundary regularity results such as
\cite[Theorems~1.6, 1.7, 1.9, and~1.10]{JiangXie} and
\cite[Corollary~1.8]{JiangXie}, as well as \cite[Theorem~7.4]{JiangXie}, 
remain valid with condition~(1.10) removed.  
\end{remark}
For clarity, we record the principal upgrades to the revised results of
\cite{JiangXie}.  In each item below, the mass-growth condition~(1.10) is
deleted.
\begin{enumerate}
 \item The finite-regularity conclusions of
 \cite[Theorems~1.6 and~1.9]{JiangXie} hold without an a priori mass-growth
 bound.  In particular, $C^{n,\alpha}$ boundary data give $C^{n,\alpha}$
 regularity up to the boundary, while $C^{l,\alpha}$ data give the finite
 logarithmic expansion stated in Theorem~1.9.

 \item For smooth boundary data, \cite[Theorem~1.7]{JiangXie} gives a smooth
 dependence on $y'$, $y^n$, and $y^n\log y^n$; when $n$ is even, the logarithmic
 terms vanish and the graph is smooth up to the boundary.  In dimension
 $n=3$, the Willmore criterion in \cite[Corollary~1.8]{JiangXie} is upgraded
 in the same way.

 \item The convergence results \cite[Theorems~1.10 and~7.4]{JiangXie} also
 require no mass-growth hypothesis.  Thus analytic boundary data yield the
 convergent analytic representation in $y$ and $y^n\log y^n$, while smooth
 boundary data yield the absolutely and uniformly convergent expansion in
 $-(y^n)^n\log y^n$ described in Theorem~7.4.
\end{enumerate}

\begin{remark}
\label{rem:CM}
The doubling principle and the organization of the proof come from
Colding--Minicozzi \cite{ColdingMinicozzi}: confinement controls transverse
tilt, which in turn yields volume doubling.  The role of
\cref{thm:hyperbolic-comparison} is to supply the parallel fixed-scale statement in
hyperbolic space for the application above.  The required modification is
the corrected Green barrier, which compensates for the identity
$\nabla^2z_\alpha=z_\alpha g_{\Hyp}$ and for the curvature term in the radial
Hessian.  We do not claim that the doubling mechanism itself is new.
\end{remark}

The remainder of the paper is organized as follows.  In
\cref{sec:preliminaries} we fix the hyperbolic and geometric-measure-theory
conventions and isolate the input from \cite{JiangXie}.  In
\cref{sec:technical} we record the Euclidean input, adapt the corresponding
fixed-scale estimate to hyperbolic space, and prove the squashing estimates.  The
mass bounds are then proved in \cref{sec:proofs}. In Section \ref{sec:app}, we present further applications of Theorem \ref{thm:Lin-problem}.

\vspace{2em}
{\bf Disclosure on AI assistance.} The authors used AI-assisted tools, principally ChatGPT. The authors wrote and verified all theorem statements, proofs, and they take full responsibility for the contents of the paper.

\section{Preliminaries}\label{sec:preliminaries}

\subsection{Hyperbolic-space conventions}

In the upper-half-space model the hyperbolic metric and distance satisfy
\begin{equation}\label{eq:upper-distance}
 g_{\Hyp}=(x^{\mathfrak m})^{-2}|dx|^2,
 \qquad
 \cosh d_{\Hyp}(x,\widetilde x)
 =1+\frac{|x-\widetilde x|^2}
 {2x^{\mathfrak m}\widetilde x^{\mathfrak m}}.
\end{equation}
The upper-half-space and hyperboloid formulas used below are standard; see,
for example, \cite[Chapters~3--4]{Ratcliffe}.
We write $B_r(p)=B_{\Hyp^{\mathfrak m}}(p,r)$ unless another ambient
space is displayed.  An $n$-dimensional subspace $L\subset\Hyp^{\mathfrak m}$
is called \emph{totally geodesic} if every geodesic of the induced metric on
$L$ is an ambient geodesic.  Equivalently, its second fundamental form
vanishes; in the upper-half-space model, vertical Euclidean half-planes are
standard examples.  The open $\delta$-tube about $L$ is
\[
 \mathcal T_\delta(L)
 =\{x\in\Hyp^{\mathfrak m}:\dist_{\Hyp}(x,L)<\delta\}.
\]

\subsection{Currents and varifolds}\label{sec:GMT-prelim}

We use the terminology of Federer \cite{Federer} and Simon \cite{Simon}.
Throughout the paper, a rectifiable current is understood to have integer
multiplicity and locally finite mass.  Thus an $n$-current on a Riemannian
manifold $M$ has a representation
\[
 T=\cur{M_T,\tau_T,\theta_T},
\]
where $M_T$ is countably $n$-rectifiable, $\tau_T$ is an orientation, and
$\theta_T$ is integer valued.  Its mass measure and mass are denoted by
$\|T\|$ and $\Mass_M(T)$, respectively; $T\restr A$ denotes restriction to
a Borel set $A$, $f_\#T$ the pushforward when $f$ is Lipschitz and proper on
$\supp T$,
and $\supp T$ the closed support of $\|T\|$.  The current is locally integral
if both $T$ and $\partial T$ are integer rectifiable and have locally finite
mass.

We use the standard local notion of area minimization.  Namely, $T$ is
\emph{locally area minimizing} if
for every $U\Subset M$ and every integer-rectifiable $n$-current $W$ with
$\supp W\Subset U$ and $\partial W=0$,
\begin{equation}\label{eq:local-minimizer-definition}
 \Mass_U(T)\le \Mass_U(T+W),
 \qquad \Mass_U(S):=\|S\|(U).
\end{equation}
On a regular relatively compact domain this is equivalent to minimizing
$T\restr U$ among integer-rectifiable currents with the same boundary; see
\cite{Federer,Simon}.  

An integral $n$-varifold $V$ in $M$ is a Radon measure on the Grassmann
bundle $G_n(M)$ induced by a countably rectifiable set with
integer-valued density.  Its weight measure is $\mu_V=\|V\|$.  It is
\emph{stationary} in an open set $U$ if
\begin{equation}\label{eq:stationarity-definition}
 \delta V(X)=\int_{G_n(U)}\operatorname{div}_S X(x)\,dV(x,S)=0
\end{equation}
for every $X\in C_c^1(U;TM)$.  The varifold associated with
$T=\cur{M_T,\tau_T,\theta_T}$ is
\[
 V_T=\mathbf v(M_T,|\theta_T|);
\]
thus $\mu_{V_T}=\|T\|$, but the orientation is forgotten.  In particular,
the algebraic multiplicity of a projected current need not equal the density
of $V_T$.  If $T$ is locally area minimizing and
$(\partial T)\restr U=0$, then $V_T$ is a stationary integral varifold in
$U$; this is the standard first variation of mass
\cite[Sections~16--17]{Simon}.  No stationarity assertion is made across
$\supp(\partial T)$.

Whenever we say that a projection has algebraic multiplicity $Q$, we mean a
current identity of the form $f_\#T=Q\cur{D}$ on an oriented target domain
$D$.  Contributions from different sheets enter this identity with their
orientation signs.  The associated varifold, on the other hand, assigns
positive weight $|\theta_T|$ to every sheet.

We shall also use the following standard facts.  If $f$ is $1$-Lipschitz and
$T$ is locally integral, then for almost every $r$ the slice
$\langle T,f,r\rangle$ is integral and
\begin{equation}\label{eq:slicing-facts}
 \partial(T\restr\{f<r\})
 = (\partial T)\restr\{f<r\}+\langle T,f,r\rangle,
 \qquad
 \Mass(\langle T,f,r\rangle)
 \le \frac{d}{dr}\Mass(T\restr\{f<r\})
\end{equation}
for almost every $r$.   If $H:[0,1]\times M\to N$ is Lipschitz, one has
\begin{equation}\label{eq:homotopy-formula}
 \partial H_\#(\cur{[0,1]}\times T)
 =H(1,\cdot)_\#T-H(0,\cdot)_\#T
  -H_\#(\cur{[0,1]}\times\partial T).
\end{equation}
The corresponding product-current mass estimate bounds the swept mass by
the displacement times the appropriate spatial Jacobian.  Finally, the
Constancy Theorem says that a boundaryless top-dimensional integral current
on a connected oriented manifold is an integer multiple of its fundamental
current.  These statements are proved in
\cite[Sections~4.1 and~4.3]{Federer}; the Constancy Theorem is
\cite[4.1.7]{Federer}.

\subsection{Localization near the asymptotic boundary}
\label{sec:boundary-prelim}

Fix $Q_0\in\Gamma$ and translate so that $Q_0=(\mathbf0,0)$.  With the notation
of \cite[(1.3)--(1.7)]{JiangXie}, let
\begin{align}
 H^+&=\{(x',x^{\mathfrak m}):(x',0)\in T_{Q_0}\Gamma,
                    \ x^{\mathfrak m}>0\},\label{eq:Hplus}\\
 G_R&=\{(x',x^{\mathfrak m}):|x'|<R,
                    \ 0<x^{\mathfrak m}<R\},\label{eq:GR}\\
 B_R^+&=G_R\cap H^+,\label{eq:BRplus}
\end{align}
and let $\Proj$ denote Euclidean orthogonal projection onto $H^+$.  We say,
as in \cite[Section~2]{JiangXie}, that $T$ is \emph{asymptotic to} $\Gamma$
if $\partial T=0$ in $\Hyp^{\mathfrak m}$ and the Euclidean boundary of
$\supp T$ is $\Gamma$.

The following proposition isolates exactly the part of the earlier paper
used in \cref{sec:proofs}.  It contains no mass estimate.

\begin{proposition}\label{prop:boundary-localization}
Let $T$ and $\Gamma$ satisfy the hypotheses of
\cref{thm:remove-growth}, and let $\alpha\in(0,1]$ be the
Hölder exponent for which
$\Gamma\cap\overline{G_R}\in C^{1,\alpha}$.  There are $R_0\in(0,R)$,
$t_0>0$, and $C_\Gamma<\infty$ such that for every
$P\in\supp T\cap G_{R_0}$ with $t=x^{\mathfrak m}(P)<t_0$, the isometry
\begin{equation}\label{eq:boundary-rescaling}
 \Phi_P(x',x^{\mathfrak m})
 =t^{-1}(x'-x'(P),x^{\mathfrak m})
\end{equation}
has the following properties.  If $\widehat T_P=(\Phi_P)_\#T$ and
$q_P=\Phi_P(P)$, there are a vertical totally geodesic $n$-plane $L_P$, its
height-preserving projection $\Pi_P$, and $p_P=\Pi_P(q_P)$ such that
\begin{align}
 (\partial\widehat T_P)\restr B_2(q_P)&=0,
 \label{eq:rescaled-no-boundary}\\
 \dist_{\Hyp}(q_P,L_P)+
 \sup_{x\in\supp\widehat T_P\cap B_2(q_P)}\dist_{\Hyp}(x,L_P)
 &\le C_\Gamma t^\alpha,\label{eq:rescaled-tube}\\
 (\Pi_P)_\#(\widehat T_P\restr B_2(q_P))
 \restr D_{3/2}^{L_P}(p_P)
 &=\cur{D_{3/2}^{L_P}(p_P)}.\label{eq:rescaled-projection}
\end{align}
\end{proposition}

\begin{proof}
The first identity follows from $\partial T=0$ in $\Hyp^{\mathfrak m}$ and
invariance of boundary under pushforward.  Apply the convex-hull estimate
\cite[(2.8)]{JiangXie} on a slightly larger fixed normalized ball.  After
\eqref{eq:boundary-rescaling}, its quantitative form
\cite[(3.30)]{JiangXie} gives
\eqref{eq:rescaled-tube}.

We spell out the localization of the projection because it is independent of
any mass bound.  By \eqref{eq:boundary-projection} and the Constancy Theorem, the
multiplicity-one projection identity persists on every smaller cylinder
compactly contained in the cylinder where \eqref{eq:boundary-projection} holds.
After applying $\Phi_P$, its target contains
$D_{7/4}^{L_P}(p_P)$ and $B_2(q_P)$ lies in the normalized image of that
smaller cylinder once $R_0$ and $t_0$ are sufficiently small.  The same
convex-hull estimate shows that every point in the support of the normalized
restriction whose projection lies in $D_{3/2}^{L_P}(p_P)$ belongs to
$B_2(q_P)$: both its transverse displacement from $L_P$ and the displacement
of $q_P$ from $L_P$ are $O_\Gamma(t^\alpha)$, while the two radii leave a
fixed margin.  Testing the projected currents against forms compactly
supported in $D_{3/2}^{L_P}(p_P)$ therefore permits restriction before
pushforward and gives \eqref{eq:rescaled-projection}.  This is the
localized identity used in \cite[(3.15)]{JiangXie}.  None of these steps
invokes
\cite[(1.10)]{JiangXie}.
\end{proof}

\subsection{The hyperboloid model and transverse identities}
\label{sec:hyperboloid}

Realize hyperbolic space as
\begin{equation}\label{eq:hyperboloid}
 \Hyp^{\mathfrak m}=\bigl\{X\in\R^{\mathfrak m,1}:
 \langle X,X\rangle_L=-1,\ X_0>0\bigr\},
\end{equation}
where
\[
 \langle X,Y\rangle_L=-X_0Y_0+\sum_{j=1}^{\mathfrak m}X_jY_j.
\]
After an ambient isometry,
\begin{equation}\label{eq:L-model}
 L=\Hyp^{\mathfrak m}\cap\{X_{n+1}=\cdots=X_{n+k}=0\}
 \cong\Hyp^n.
\end{equation}
For $1\le\alpha\le k$, define
\[
 z_\alpha(X)=X_{n+\alpha},
 \qquad z=(z_1,\dots,z_k),
 \qquad \zeta=|z|^2.
\]
If $\rho(X)=\dist_{\Hyp}(X,L)$, the standard normal-coordinate formula in the
hyperboloid model gives
\begin{equation}\label{eq:zeta-rho}
 \zeta=\sinh^2\rho.
\end{equation}

Fix $p\in L$ and put $r(X)=\dist_{\Hyp}(p,X)$.  At
$X\in\Hyp^{\mathfrak m}$ and for an $n$-plane
$S\subset T_X\Hyp^{\mathfrak m}$, set
\begin{equation}\label{eq:E-a}
 E(X,S)=\sum_{\alpha=1}^k|P_S\nabla z_\alpha|^2,
 \qquad
 a(X,S)=|P_{S^\perp}\nabla r|^2.
\end{equation}
Here $P_S$ and $P_{S^\perp}$ denote orthogonal projection onto $S$ and its
orthogonal complement, respectively.  Thus $E$ is the transverse tilt and
$a$ is the radial normal tilt.  For a smooth function $u$, write
\[
L_Su=\tr_S\nabla^2u.
\]

The Gauss formula for the hyperboloid model gives, for
$\alpha,\beta\in\{1,\dots,k\}$,
\begin{align}
 \nabla z_\alpha&=e_{n+\alpha}+z_\alpha X, \label{eq:grad-z}\\
 \nabla^2z_\alpha&=z_\alpha g_{\Hyp}, \label{eq:hess-z}\\
 \langle\nabla z_\alpha,\nabla z_\beta\rangle
 &=\delta_{\alpha\beta}+z_\alpha z_\beta. \label{eq:gram-z}
\end{align}
These are standard coordinate identities; see
\cite[Chapter~3]{Ratcliffe}.  Differentiating $\zeta=\sum_\alpha z_\alpha^2$
and tracing over $S$ therefore gives
\begin{align}
 L_S\zeta&=2E+2n\zeta, \label{eq:Lzeta}\\
 |\nabla \zeta|^2&=4\zeta(1+\zeta). \label{eq:grad-zeta}
\end{align}

\section{Technical Lemmas}\label{sec:technical}

\subsection{The Euclidean confined-volume theorem}

We begin with the Euclidean theorem of Colding--Minicozzi, which is the source
of the doubling mechanism used here.  If
$\pi^\perp:\R^{n+k}\to\R^k$ is orthogonal projection, the set
$\{|\pi^\perp x|\le\delta R\}$ is a slab of width $2\delta R$ about the
$n$-plane $\R^n\times\{0\}$.  

\begin{theorem}[Colding--Minicozzi \cite{ColdingMinicozzi}]
\label{thm:CM-doubling}
There are $\delta_{\rm CM}=\delta_{\rm CM}(n)>0$ and
$C_{\rm CM}=C_{\rm CM}(n)<\infty$ with the following property.  Let $V$ be a
proper stationary integral $n$-varifold in $\R^{n+k}$.  If
\[
 B_{4R}(0)\cap\supp V
 \subset\{x\in\R^{n+k}:|\pi^\perp x|\le\delta_{\rm CM}R\},
\]
then
\[
 \mu_V(B_{2R}(0))\le C_{\rm CM}\mu_V(B_R(0)).
\]
The constants do not depend on $k$.
\end{theorem}

This is \cite[Theorem~0.5]{ColdingMinicozzi}; its proof for $n\ge3$ is given
in \cite[Section~1]{ColdingMinicozzi}, while the surface case is treated in
\cite[Section~6.1]{ColdingMinicozzi}.  We use the following localized form
for \cref{thm:Lin-problem}.

\begin{corollary}[Localized Euclidean doubling]
\label{cor:CM-local}
Let $p\in L$, where $L\subset\R^{n+k}$ is an affine $n$-plane.  If $V$ is
an integral $n$-varifold stationary in $B_{4R}(p)$ and
\[
 \supp V\cap B_{4R}(p)
 \subset\{x:\dist(x,L)\le\delta_{\rm CM}R\},
\]
then
\begin{equation}\label{eq:CM-local-doubling}
 \mu_V(B_{2R}(p))\le C_{\rm CM}\mu_V(B_R(p)).
\end{equation}
\end{corollary}

\begin{proof}
Translate $p$ to the origin, rotate the affine plane to
$\R^n\times\{0\}$, and dilate by $R^{-1}$.  Repeat the proof of
\cite[Theorem~0.5]{ColdingMinicozzi} at this scale.  For $n\ge3$, the
first-variation identities and the covering reduction in
\cite[Section~1]{ColdingMinicozzi} use cutoff fields compactly supported in
the scale-four region.  The proof therefore uses only the restriction of the
varifold to that region and stationarity there.  The surface argument in
\cite[Section~6.1]{ColdingMinicozzi} has the same localization, with $\log s$
in place of $s^{2-n}$.  Approximating the piecewise linear cutoffs by smooth
compactly supported cutoffs gives the assertion for a varifold defined and
stationary only in $B_4(0)$.  Scaling back proves
\eqref{eq:CM-local-doubling}.
\end{proof}

The second Euclidean ingredient is the algebraic-multiplicity-$Q$ form of the
squashing estimate.  Here $Q$ is the degree of the projected current, not a
prescribed density of the associated varifold.  Throughout this Euclidean
subsection, balls, distances, and masses are Euclidean.  For an oriented
affine $n$-plane $L$ and $p\in L$,
write
$D_{r,{\rm E}}^L(p)=B_r(p)\cap L$, and let $\mathbf p_L$ be orthogonal
projection onto the affine plane $L$.  The subscript ${\rm E}$ distinguishes
these disks from the hyperbolic disks $D_s^L(p)$ used elsewhere. The following result is essentially Lemma 3.1 in \cite{JiangXie}.

\begin{lemma}[Euclidean squashing with algebraic multiplicity]
\label{lem:Euclidean-squashing}
There is $c_E=c_E(n)<\infty$ with the following property.  Let $T$ be an
area-minimizing integral $n$-current in $\R^{n+k}$, let $p\in L$, and assume
\begin{align}
 (\partial T)\restr B_{2R}(p)&=0,\label{eq:E-squash-boundary}\\
 \supp T\cap B_{2R}(p)&\subset\{\dist(\cdot,L)\le\theta R\},
 \label{eq:E-squash-tube}\\
 (\mathbf p_L)_\#(T\restr B_{2R}(p))
 \restr D_{(2-2\theta)R,{\rm E}}^L(p)
 &=Q\cur{D_{(2-2\theta)R,{\rm E}}^L(p)}
 \label{eq:E-squash-projection}
\end{align}
for an integer $Q\ge1$ and $0<\theta<1/8$.  Then
\begin{equation}\label{eq:E-squash-conclusion}
 \Mass(T\restr B_R(p))
 \le \exp\!\left(-\frac{1}{c_E\theta}\right)
       \Mass(T\restr B_{2R}(p))+c_EQR^n.
\end{equation}
The constant is independent of $k$ and $Q$.
\end{lemma}

\begin{proof}
Translation and dilation reduce the proof to $p=0$ and $R=1$.  Put
$F(r)=\Mass(T\restr B_r)$ for $1\le r\le2$.  For almost every such $r$,
slicing gives the cycle $Z_r=\partial(T\restr B_r)$ with
$\Mass(Z_r)\le F'(r)$.  Set $s=r-2\theta$ and abbreviate
$D_{u,{\rm E}}^L=D_{u,{\rm E}}^L(0)$.

Let
\[
 \mathcal H_\tau(x)=(1-\tau)x+\tau\mathbf p_Lx,
 \qquad
 A_r=\mathcal H_\#\bigl(\cur{[0,1]}\times Z_r\bigr).
\]
This straight-line homotopy has speed at most $\theta$ and spatial
differential at most one.  The homotopy formula and the product-current mass
estimate give
\[
 \partial A_r=(\mathbf p_L)_\#Z_r-Z_r,
 \qquad
 \Mass(A_r)\le\theta\Mass(Z_r).
\]
Moreover, every point of
$\supp(T\restr B_2)$ projecting into $D_{s,{\rm E}}^L$ lies in $B_r$, by
\eqref{eq:E-squash-tube}.  Thus \eqref{eq:E-squash-projection} localizes to
\[
 (\mathbf p_L)_\#(T\restr B_r)\restr D_{s,{\rm E}}^L
 =Q\cur{D_{s,{\rm E}}^L}.
\]
By locality of the boundary inside the open disk $D_{s,{\rm E}}^L$,
$(\mathbf p_L)_\#Z_r$ has no support there; its support is contained in
$\overline{D_{r,{\rm E}}^L}\setminus D_{s,{\rm E}}^L$.  Consequently, the
top-dimensional current
\[
 C_r=(\mathbf p_L)_\#(T\restr B_r)-Q\cur{D_{s,{\rm E}}^L}
\]
is supported in $\overline{D_{r,{\rm E}}^L}\setminus D_{s,{\rm E}}^L$ and has
boundary $(\mathbf p_L)_\#Z_r-Q\cur{\partial D_{s,{\rm E}}^L}$.  Radial
retraction of this annulus onto $\partial D_{s,{\rm E}}^L$ annihilates $C_r$,
since the terminal map has
rank at most $n-1$.  It follows that the same retraction sends
$(\mathbf p_L)_\#Z_r$ to $Q\cur{\partial D_{s,{\rm E}}^L}$.  Explicitly, on
the annulus set
\[
 \mathcal R_\tau(y)=
 \left((1-\tau)+\tau\frac{s}{|y|}\right)y,
 \qquad 0\le\tau\le1,
\]
and set
\[
 K_r=\mathcal R_\#\bigl(
       \cur{[0,1]}\times(\mathbf p_L)_\#Z_r\bigr).
\]
Its displacement is at most $2\theta$ and its spatial differential is at
most one.  Thus
\[
 \partial K_r=Q\cur{\partial D_{s,{\rm E}}^L}-(\mathbf p_L)_\#Z_r,
 \qquad
 \Mass(K_r)\le2\theta\Mass(Z_r).
\]
Consequently,
\[
 R_r=-A_r-K_r+Q\cur{D_{s,{\rm E}}^L}
\]
satisfies
\[
 \partial R_r=Z_r,
 \qquad
 \Mass(R_r)\le c_E\theta\Mass(Z_r)+c_EQ.
\]
To use this filling as a competitor, set
\[
 W_r=R_r-(T\restr B_r).
\]
Then $\partial W_r=0$ and $W_r$ is compactly supported in $B_2$.  Choose a
bounded open set $U$ containing this support.  For the almost-everywhere
radii under consideration, $\|T\|(\partial B_r)=0$.  Area minimality and the
triangle inequality give
\[
 \Mass_U(T)
 \le \Mass_U(T+W_r)
 \le \Mass_U(T\restr(U\setminus B_r))+\Mass(R_r).
\]
Subtracting the first term on the right gives $F(r)\le\Mass(R_r)$.  Hence
$F(r)\le c_E\theta F'(r)+c_EQ$.  Applying the integrating factor to
$F-c_EQ$ on $[1,2]$ yields \eqref{eq:E-squash-conclusion}.
\end{proof}

For the hyperbolic application we follow the fixed-scale scheme of
Colding--Minicozzi and replace its Euclidean tilt and radial identities by the
curvature-corrected estimates in the following four subsections.

\subsection{A codimension-free tilt comparison}\label{sec:tilt}

Set
\[
 N_\alpha=\nabla z_\alpha,
 \qquad
 \mathcal N_X=\operatorname{span}\{N_1,\dots,N_k\},
 \qquad
 \mathcal H_X=\mathcal N_X^\perp.
\]
By \eqref{eq:gram-z}, $\dim\mathcal N_X=k$ and $\dim\mathcal H_X=n$.

\begin{lemma}[Radial normal tilt versus transverse tilt]\label{lem:tilt-comparison}
For every $r_0>0$ and every $X$ with $r(X)\ge r_0$,
\begin{equation}\label{eq:tilt-comparison}
 a(X,S)\le2E(X,S)+2\coth^2(r_0)\zeta(X).
\end{equation}
In particular, on each fixed annulus $r_0\le r\le r_1$,
\[
 a\le C(r_0)(E+\zeta),
\]
with a constant independent of $k$.
\end{lemma}

\begin{proof}
Define the positive semidefinite endomorphism
\[
 A=\sum_{\alpha=1}^kN_\alpha\otimes N_\alpha.
\]
The nonzero eigenvalues of $A$ are the eigenvalues of the Gram matrix
$I+zz^T$: they are $1$ with multiplicity $k-1$ and $1+\zeta$ with
multiplicity one.  Hence
\begin{equation}\label{eq:A-PN}
 A\ge P_{\mathcal N_X}.
\end{equation}
It follows that
\begin{equation}\label{eq:E-trace}
 E=\tr(P_SA)\ge\tr(P_SP_{\mathcal N_X}).
\end{equation}

Let $f=\cosh r=-\langle X,p\rangle_L$.  Since $p\in L$,
\[
 \nabla f=-p+fX,
 \qquad
 \langle\nabla f,\nabla z_\alpha\rangle=fz_\alpha.
\]
Dividing by $|\nabla f|=\sinh r$ gives
\begin{equation}\label{eq:r-Nalpha}
 \langle\nabla r,N_\alpha\rangle=\coth r\,z_\alpha.
\end{equation}
Therefore
\begin{align}
 |P_{\mathcal N_X}\nabla r|^2
 &=\coth^2r\,z^T(I+zz^T)^{-1}z \notag\\
 &=\coth^2r\,\frac{\zeta}{1+\zeta}
 \le\coth^2(r_0)\zeta. \label{eq:PN-r}
\end{align}

Put $v=P_{\mathcal H_X}\nabla r$.  Since $\dim\mathcal H_X=\dim S=n$,
\begin{equation}\label{eq:trace-complement}
 \tr(P_{\mathcal H_X}P_{S^\perp})
 =\tr(P_{\mathcal N_X}P_S).
\end{equation}
The largest eigenvalue of the positive semidefinite map
$P_{\mathcal H_X}P_{S^\perp}P_{\mathcal H_X}$ is bounded by its trace.
Since $|v|\le1$, \eqref{eq:E-trace} and
\eqref{eq:trace-complement} imply
\begin{equation}\label{eq:PH-r}
 |P_{S^\perp}v|^2
 \le\tr(P_{\mathcal H_X}P_{S^\perp})
 =\tr(P_{\mathcal N_X}P_S)
 \le E.
\end{equation}
Combining \eqref{eq:PN-r}, \eqref{eq:PH-r}, and
$|u+v|^2\le2|u|^2+2|v|^2$ proves \eqref{eq:tilt-comparison}.
\end{proof}

\subsection{The curvature-corrected Green-function barrier}\label{sec:barrier}

For $n\ge2$, define
\begin{equation}\label{eq:G-def}
 G_n(r)=\int_r^\infty(\sinh t)^{1-n}\,dt,
 \qquad
 \mathfrak q_n(r)=-G_n'(r)=(\sinh r)^{1-n}>0,
\end{equation}
and
\begin{equation}\label{eq:A-n}
 A_n(r)=n\mathfrak q_n(r)\coth r.
\end{equation}
Then
\begin{equation}\label{eq:G-ode}
 G_n''+(n-1)\coth r\,G_n'=0.
\end{equation}

The standard distance-Hessian formula in hyperbolic space is
\[
 \nabla^2r=\coth r\,(g_{\Hyp}-dr\otimes dr);
\]
see \cite[Section~5.5]{Petersen}.  Thus, if
$b=|P_S\nabla r|^2=1-a$, then \eqref{eq:G-ode} gives
\begin{equation}\label{eq:LG}
 \begin{split}
 L_SG_n(r)
 &=G_n''b+G_n'\coth r\,(n-b)\\
 &=-A_n(r)a.
 \end{split}
\end{equation}

Fix
\begin{equation}\label{eq:I-annulus}
 I=\left[\frac1{16},\frac32\right].
\end{equation}
By \cref{lem:tilt-comparison}, there is $C_a=C_a(n)$ such that
\begin{equation}\label{eq:a-Ezeta-I}
 a\le C_a(E+\zeta)
 \qquad\text{on }I.
\end{equation}
Let $A_*=\max_I A_n$ and choose
\begin{equation}\label{eq:Lambda-choice}
 \Lambda\ge\frac{A_*C_a+1}{2}.
\end{equation}
Define
\begin{equation}\label{eq:h-def}
 h=G_n(r)+\Lambda\zeta.
\end{equation}

\begin{lemma}[Subharmonic corrected barrier]\label{lem:subharmonic-h}
On the annulus \eqref{eq:I-annulus},
\begin{equation}\label{eq:Lh}
 L_Sh\ge E+\zeta.
\end{equation}
\end{lemma}

\begin{proof}
By \eqref{eq:Lzeta}, \eqref{eq:LG}, and \eqref{eq:a-Ezeta-I},
\begin{align*}
 L_Sh
 &=-A_n(r)a+2\Lambda E+2n\Lambda\zeta\\
 &\ge(2\Lambda-A_*C_a)E+(2n\Lambda-A_*C_a)\zeta
 \ge E+\zeta.
\end{align*}
\end{proof}

Since $G_n:(0,\infty)\to(0,\infty)$ is strictly decreasing, define $\bar r$ by
\begin{equation}\label{eq:rbar-def}
 G_n(\bar r)=h=G_n(r)+\Lambda\zeta.
\end{equation}
Thus $\bar r\le r$.  Away from $p$, the inverse-function theorem also
shows that $\bar r$ is smooth; this is the only region on which its
derivatives will be used.

\begin{lemma}[Comparison of $r$ and $\bar r$]\label{lem:r-rbar}
There exist $\eps_0=\eps_0(n)>0$, $\sigma=\sigma(n)>0$, and $C=C(n)$
such that, whenever $\rho\le\eps_0$ and
\[
 \frac1{16}\le\bar r\le\frac{17}{12},
\]
one has
\begin{equation}\label{eq:r-rbar-est}
 r\le\frac32-\sigma,
 \qquad
 0\le r-\bar r\le\frac1{24},
 \qquad
 |\nabla\bar r|\le C.
\end{equation}
\end{lemma}

\begin{proof}
Set $\mathfrak q_*=\min_I\mathfrak q_n>0$ and
\[
 d_*=G_n\!\left(\frac{17}{12}\right)-G_n\!\left(\frac32\right)>0.
\]
Choose $\eps_0$ so small that
\begin{equation}\label{eq:eps-choice}
 \Lambda\sinh^2\eps_0
 <\min\left\{\frac{d_*}{2},\frac{\mathfrak q_*}{24}\right\}.
\end{equation}
By \eqref{eq:zeta-rho}, $\zeta\le\sinh^2\eps_0$.  From
\eqref{eq:rbar-def},
\[
 G_n(r)=G_n(\bar r)-\Lambda\zeta
 \ge G_n\!\left(\frac{17}{12}\right)-\frac{d_*}{2}
 >G_n\!\left(\frac32\right).
\]
Because $G_n$ is strictly decreasing, there is a fixed
$\sigma=\sigma(n)>0$ for which $r\le3/2-\sigma$.

The interval between $\bar r$ and $r$ lies in $I$, and hence
\[
 \Lambda\zeta=G_n(\bar r)-G_n(r)
 =\int_{\bar r}^{r}\mathfrak q_n(s)\,ds
 \ge \mathfrak q_*(r-\bar r).
\]
The second term in \eqref{eq:eps-choice} gives $r-\bar r\le1/24$.

Differentiating \eqref{eq:rbar-def},
\begin{equation}\label{eq:grad-rbar}
 G_n'(\bar r)\nabla\bar r
 =G_n'(r)\nabla r+\Lambda\nabla\zeta.
\end{equation}
Both $r$ and $\bar r$ lie in the compact interval $I$.  The function
$\mathfrak q_n$ is
bounded above and below there, while \eqref{eq:grad-zeta} and
$\zeta\le\sinh^2\eps_0$ bound $|\nabla\zeta|$.  Equation
\eqref{eq:grad-rbar} proves the gradient estimate.
\end{proof}

\begin{lemma}\label{lem:target-annulus}
After decreasing $\eps_0(n)$ if necessary, if $\rho\le\eps_0$ and
\[
 \frac14\le r\le\frac54,
\]
then
\begin{equation}\label{eq:rbar-target}
 \frac5{24}\le\bar r\le\frac54.
\end{equation}
\end{lemma}

\begin{proof}
The upper bound follows from $\bar r\le r$.  Suppose that $\bar r<1/16$.  Then
\begin{align*}
 \Lambda\zeta
 &=G_n(\bar r)-G_n(r)\\
 &>G_n\!\left(\frac1{16}\right)-G_n\!\left(\frac14\right)
 =\int_{1/16}^{1/4}\mathfrak q_n(s)\,ds
 \ge\frac3{16}\mathfrak q_*,
\end{align*}
contrary to \eqref{eq:eps-choice}.  Hence $\bar r\ge1/16$, so \cref{lem:r-rbar} applies and gives
\[
 \bar r\ge r-\frac1{24}\ge\frac14-\frac1{24}=\frac5{24}.
\]
\end{proof}

\subsection{The transverse-energy estimate}\label{sec:transverse}

Let $V$ satisfy the hypotheses of \cref{thm:hyperbolic-comparison}, with
$\eps_D\le\eps_0$.  Choose a smooth function
$\eta:[0,\infty)\to[0,1]$ such that
\begin{equation}\label{eq:eta}
\begin{aligned}
 &\eta=0 &&\text{on }[0,1/16],
 &\qquad \eta'\ge0 &&\text{on }[1/16,1/8],\\
 &\eta=1 &&\text{on }[1/8,4/3],
 &\qquad \eta'\le0 &&\text{on }[4/3,17/12],\\
 &\eta=0 &&\text{on }[17/12,\infty),
 &\qquad |\eta'|&\le C(n).
\end{aligned}
\end{equation}

\begin{lemma}[Transverse energy on the target annulus]\label{lem:transverse-energy}
Under the assumptions of \cref{thm:hyperbolic-comparison},
\begin{equation}\label{eq:transverse-energy}
 \int_{B_{5/4}(p)\setminus B_{1/4}(p)}(E+\zeta)\,dV
 \le C(n)\,\mu_V\bigl(B_{1/4}(p)\bigr).
\end{equation}
\end{lemma}

\begin{proof}
Let $\sigma$ be as in \cref{lem:r-rbar}.  Choose a smooth radial
cutoff $\chi_0:[0,\infty)\to[0,1]$ such that
\[
 \chi_0=1\quad\text{on }[0,3/2-\sigma/2],
 \qquad
 \supp\chi_0\subset[0,3/2).
\]
Consider
\begin{equation}\label{eq:X1}
 X_1=\chi_0(r)\eta(\bar r)\nabla h.
\end{equation}
Although $G_n(r)$ is singular at $p$, one has $\bar r\to0$ as $r\to0$
(recall that $\zeta(p)=0$), so the factor $\eta(\bar r)$ vanishes in a
neighborhood of $p$.  Hence $X_1$ extends smoothly by zero across $p$ and
is compactly supported.  Whenever $\eta(\bar r)$ or
$\eta'(\bar r)$ is nonzero, \cref{lem:r-rbar} gives
$r\le3/2-\sigma$.  Hence $\chi_0(r)=1$ and
$\nabla(\chi_0(r))=0$ at every point of
$\supp V$ where the cutoff contributes.  Thus the radial cutoff creates no
additional first-variation term.

Since $h=G_n(\bar r)$,
\[
 \nabla h=G_n'(\bar r)\nabla\bar r.
\]
Stationarity gives
\begin{equation}\label{eq:first-var-h}
 0=\int\left\{
 \eta(\bar r)L_Sh
 +\eta'(\bar r)G_n'(\bar r)|P_S\nabla\bar r|^2
 \right\}\,dV.
\end{equation}

By \cref{lem:target-annulus}, on $B_{5/4}(p)\setminus B_{1/4}(p)$ one has
\[
 \frac5{24}\le\bar r\le\frac54<\frac43.
\]
Therefore $\eta(\bar r)=1$ on the target annulus.  Moreover, wherever
$\eta(\bar r)>0$, \cref{lem:r-rbar} places $r$ in the annulus $I$, so
$L_Sh\ge E+\zeta\ge0$ there by \cref{lem:subharmonic-h}.  Consequently,
\begin{equation}\label{eq:target-Lh}
 \int_{B_{5/4}\setminus B_{1/4}}(E+\zeta)\,dV
 \le\int\eta(\bar r)L_Sh\,dV.
\end{equation}

Split the $\eta'$ term in \eqref{eq:first-var-h} into the inner and outer
transition regions.  On the outer transition, $\eta'\le0$ and $G_n'<0$, so
\[
 \eta'G_n'|P_S\nabla\bar r|^2\ge0.
\]
This term has the favorable sign and may be discarded.  On the inner
transition, $1/16<\bar r<1/8$, and \cref{lem:r-rbar} gives
\[
 r\le\frac18+\frac1{24}=\frac16<\frac14.
\]
Moreover, $|\nabla\bar r|\le C(n)$, and $G_n'(\bar r)$ is uniformly bounded
on the transition interval.  Thus
\begin{align*}
 \int\eta(\bar r)L_Sh\,dV
 &\le-\int_{\{1/16<\bar r<1/8\}}
 \eta'(\bar r)G_n'(\bar r)|P_S\nabla\bar r|^2\,dV\\
 &\le C(n)\mu_V\bigl(B_{1/4}(p)\bigr).
\end{align*}
Combining this with \eqref{eq:target-Lh} proves \eqref{eq:transverse-energy}.
\end{proof}

\subsection{The radial estimate}\label{sec:radial}

Choose the piecewise linear cutoff
\begin{equation}\label{eq:psi}
 \psi(r)=
 \begin{cases}
 0,&0\le r\le1/8,\\
 8r-1,&1/8<r<1/4,\\
 5/4-r,&1/4\le r<5/4,\\
 0,&r\ge5/4.
 \end{cases}
\end{equation}
It may be approximated by smooth cutoffs; the identities below pass to the limit.

\begin{lemma}[Tangential radial energy]\label{lem:radial-energy}
Under the assumptions of \cref{thm:hyperbolic-comparison},
\begin{equation}\label{eq:radial-energy}
 \int_{B_{5/4}(p)\setminus B_{1/4}(p)}|P_S\nabla r|^2\,dV
 \le C(n)\,\mu_V\bigl(B_{1/4}(p)\bigr).
\end{equation}
\end{lemma}

\begin{proof}
Use the compactly supported vector field
\[
 X_2=\psi(r)\nabla G_n(r).
\]
Since $\psi$ vanishes near $r=0$, this field extends smoothly by zero across
$p$.
Write $b=|P_S\nabla r|^2=1-a$.  By \eqref{eq:LG} and stationarity,
\begin{equation}\label{eq:first-var-X2}
 0=-\int\psi(r)A_n(r)a\,dV
   -\int\psi'(r)\mathfrak q_n(r)b\,dV.
\end{equation}
Since $\psi'=8$ on $(1/8,1/4)$ and $\psi'=-1$ on $(1/4,5/4)$, \eqref{eq:first-var-X2} becomes
\begin{equation}\label{eq:radial-balance}
 \int_{B_{5/4}\setminus B_{1/4}}\mathfrak q_n(r)b\,dV
 =8\int_{B_{1/4}\setminus B_{1/8}}\mathfrak q_n(r)b\,dV
  +\int\psi(r)A_n(r)a\,dV.
\end{equation}
The first term on the right is bounded by $C(n)\mu_V(B_{1/4})$.  For the
last term, use $a\le1$ on $B_{1/4}\setminus B_{1/8}$ and
\cref{lem:tilt-comparison} on $B_{5/4}\setminus B_{1/4}$ to obtain
\begin{align*}
 \int\psi A_na\,dV
 &\le C(n)\mu_V(B_{1/4})
 +C(n)\int_{B_{5/4}\setminus B_{1/4}}(E+\zeta)\,dV\\
 &\le C(n)\mu_V(B_{1/4}),
\end{align*}
where the final inequality is \cref{lem:transverse-energy}.  Since
$\mathfrak q_n$ has a positive lower bound on $[1/4,5/4]$,
\eqref{eq:radial-balance} proves \eqref{eq:radial-energy}.
\end{proof}

\subsection{Localized squashing}\label{sec:squashing}

We now localize the squashing construction of
\cite[Lemmas~3.1 and~3.2]{JiangXie}.  We use the vertical projection
$\Pi_L$ in \eqref{eq:vertical-plane-projection}, exactly as in the
upper-half-space formulation of \cite{JiangXie}.  The next elementary
observation records the geometric estimates needed for the homotopy.

\begin{lemma}\label{lem:vertical-projection}
Let $x=(w,z,y)$ and let
$\rho(x)=\dist_{\Hyp}(x,L)$.  Then
\begin{equation}\label{eq:vertical-distance-formulas}
 \sinh\rho(x)=\frac{|z|}{y},
 \qquad
 d_{\Hyp}(x,\Pi_Lx)
 =2\operatorname{arsinh}\!\left(\frac{|z|}{2y}\right).
\end{equation}
The map $\Pi_L$ is $1$-Lipschitz.  Moreover, if $\rho(x)\le\delta\le1$,
then
\begin{equation}\label{eq:projection-displacement}
 d_{\Hyp}(x,\Pi_Lx)\le2\delta.
\end{equation}
The horizontal homotopy
\begin{equation}\label{eq:horizontal-homotopy}
 \mathcal H_\tau(w,z,y)=(w,(1-\tau)z,y),
 \qquad 0\le\tau\le1,
\end{equation}
has spatial differential of norm at most one and hyperbolic speed at most
$2\delta$ on the $\delta$-tube about $L$.
\end{lemma}

\begin{proof}
The standard upper-half-space distance formula \eqref{eq:upper-distance}
(see also \cite[Chapter~4]{Ratcliffe}), minimized over points of $L$, gives
the first identity in
\eqref{eq:vertical-distance-formulas}.  Applying the same formula to $x$
 and $(w,0,y)$ gives the second.  The differential of $\Pi_L$ is Euclidean
orthogonal projection and preserves $y$, so its norm for the conformal
metric $y^{-2}g_{\mathrm E}$ is at most one.  Finally,
\[
 2\operatorname{arsinh}\!\left(\frac{|z|}{2y}\right)
 \le\frac{|z|}{y}=\sinh\rho\le2\rho
\]
for $0\le\rho\le1$.  The assertions about \eqref{eq:horizontal-homotopy}
follow directly from its differential and from
 $|\partial_\tau\mathcal H|_{g_{\Hyp}}=|z|/y$.
\end{proof}

\begin{lemma}[Localized squashing with algebraic multiplicity]
\label{lem:squashing}
Under the assumptions of \cref{thm:local-mass}, set
\[
 F(r)=\Mass_{\Hyp^{\mathfrak m}}\bigl(T\restr B_r(q)\bigr),
 \qquad r\in[1/2,1].
\]
There is $c_s=c_s(n)<\infty$ such that, for $\delta$ sufficiently small,
\begin{equation}\label{eq:squashing-integrated}
 F(1/2)
 \le\exp\!\left(-\frac{1}{2c_s\delta}\right)F(1)+c_sQ.
\end{equation}
\end{lemma}

\begin{proof}
For almost every $r\in[1/2,1]$, slicing gives the integral cycle
\[
 Z_r=\partial\bigl(T\restr B_r(q)\bigr)
     =\langle T,d_q,r\rangle
\]
and
\begin{equation}\label{eq:slice-derivative}
 \Mass_{\Hyp}(Z_r)\le F'(r).
\end{equation}
Put $p=\Pi_L(q)$.  By \cref{lem:vertical-projection}, every point in the
support of $T\restr B_2(q)$ satisfies
\begin{equation}\label{eq:two-delta-displacements}
 d_{\Hyp}(x,\Pi_Lx)\le2\delta,
 \qquad
 d_{\Hyp}(q,p)\le2\delta.
\end{equation}
Set
\[
 s=r-4\delta.
\]
After decreasing $\delta_*$, one has $s\ge1/4$.

If $x\in\supp Z_r$ and $y=\Pi_Lx$, then
\begin{equation}\label{eq:projected-annulus}
 \bigl|d_L(y,p)-r\bigr|
 \le d_{\Hyp}(x,y)+d_{\Hyp}(q,p)
 \le4\delta.
\end{equation}
Let
\[
 S_r=(\Pi_L)_\#\bigl(T\restr B_r(q)\bigr).
\]
If $y\in D_s^L(p)$ and
$x\in\supp(T\restr B_2(q))$ satisfies $\Pi_Lx=y$, then
\[
 d_{\Hyp}(x,q)
 \le d_{\Hyp}(x,y)+d_L(y,p)+d_{\Hyp}(p,q)
 <2\delta+(r-4\delta)+2\delta=r.
\]
It follows that restriction before or after cutting down to $B_r(q)$
gives the same projected current over $D_s^L(p)$:
\[
 S_r\restr D_s^L(p)
 = (\Pi_L)_\#\bigl(T\restr B_2(q)\bigr)\restr D_s^L(p).
\]
Here the reverse inclusion is immediate from $B_r(q)\subset B_2(q)$, and
the preceding estimate proves the nontrivial inclusion of the relevant
preimages.  Since $s<3/2$, the projected-current identity
\eqref{eq:projection-local} therefore implies
\begin{equation}\label{eq:Sr-inner}
 S_r\restr D_s^L(p)=Q\,\cur{D_s^L(p)}.
\end{equation}
By locality of the boundary in the interior of $D_s^L(p)$,
$(\Pi_L)_\#Z_r=\partial S_r$ has no support there.  Moreover, $S_r$ is
supported in $D_{r+4\delta}^L(p)$.  Thus
\[
 C_r=S_r-Q\,\cur{D_s^L(p)}
\]
is supported in the closed annulus
$\overline{D_{r+4\delta}^L(p)}\setminus D_s^L(p)$ and satisfies
\begin{equation}\label{eq:boundary-Cr}
 \partial C_r=(\Pi_L)_\#Z_r-Q\,\cur{\partial D_s^L(p)}.
\end{equation}

Apply the horizontal homotopy \eqref{eq:horizontal-homotopy} to $Z_r$ and
set
\[
 A_r=\mathcal H_\#\bigl(\cur{[0,1]}\times Z_r\bigr).
\]
The homotopy formula gives
\begin{equation}\label{eq:Ar-boundary}
 \partial A_r=(\Pi_L)_\#Z_r-Z_r,
\end{equation}
and the product-current estimate together with
\cref{lem:vertical-projection} gives
\begin{equation}\label{eq:Ar-mass}
 \Mass_{\Hyp}(A_r)\le c_s\delta\,\Mass_{\Hyp}(Z_r).
\end{equation}

We next work inside $L$.  Write
$y=\exp_p(u\vartheta)$ in polar coordinates and, on the full annulus
$s\le u\le r+4\delta$ containing both $\supp C_r$ and
$\supp((\Pi_L)_\#Z_r)$, define
\begin{equation}\label{eq:radial-homotopy}
 \mathcal R_\tau\bigl(\exp_p(u\vartheta)\bigr)
 =\exp_p\bigl(((1-\tau)u+\tau s)\vartheta\bigr).
\end{equation}
Thus $\mathcal R_1$ is radial retraction to $\partial D_s^L(p)$.  Since
$C_r$ is an $n$-current in the $n$-manifold $L$ and
$\mathcal R_1$ has rank at most $n-1$,
$(\mathcal R_1)_\#C_r=0$.  Applying $\mathcal R_1$ to
\eqref{eq:boundary-Cr} yields
\begin{equation}\label{eq:radial-degree}
 (\mathcal R_1)_\#(\Pi_L)_\#Z_r
 =Q\,\cur{\partial D_s^L(p)}.
\end{equation}
Set
\[
 K_r=\mathcal R_\#\bigl(\cur{[0,1]}\times(\Pi_L)_\#Z_r\bigr).
\]
The homotopy formula and \eqref{eq:radial-degree} give
\begin{equation}\label{eq:Kr-boundary}
 \partial K_r=Q\,\cur{\partial D_s^L(p)}-(\Pi_L)_\#Z_r.
\end{equation}
The radial displacement in \eqref{eq:radial-homotopy} is at most
$8\delta$.  Its spatial differential has norm at most one: the radial factor
is $1-\tau$, while the angular factor is
\[
 \frac{\sinh((1-\tau)u+\tau s)}{\sinh u}\le1.
\]
Since $\Pi_L$ is $1$-Lipschitz,
\begin{equation}\label{eq:Kr-mass}
 \Mass_{\Hyp}(K_r)
 \le c_s\delta\,\Mass_{\Hyp}((\Pi_L)_\#Z_r)
 \le c_s\delta\,\Mass_{\Hyp}(Z_r).
\end{equation}

With compatible orientations, the current
\[
 R_r=-A_r-K_r+Q\,\cur{D_s^L(p)}
\]
satisfies $\partial R_r=Z_r$.  All currents in this construction are
supported in $B_{3/2}(q)$ when $\delta$ is sufficiently small.  Indeed,
the horizontal homotopy moves a point by at most $2\delta$, while every
point in the radial homotopy has distance at most $r+4\delta$ from $p$;
using $d_{\Hyp}(p,q)\le2\delta$ and $r\le1$ gives the assertion after fixing,
for example, $\delta_*<1/12$.  Since
$s\le1$,
\begin{equation}\label{eq:competitor-mass}
 \Mass_{\Hyp}(R_r)
 \le c_s\delta\,\Mass_{\Hyp}(Z_r)
     +Q\operatorname{Vol}_{\Hyp^n}(D_1^L(p))
 \le c_s\delta\,\Mass_{\Hyp}(Z_r)+c_sQ.
\end{equation}
The current
\[
 W_r=R_r-(T\restr B_r(q))
\]
is a cycle compactly supported in $B_{3/2}(q)$.  Since
$\|T\|(\partial B_r(q))=0$, applying
\eqref{eq:local-minimizer-definition} with this $W_r$ and then arguing as in
the Euclidean proof gives $F(r)\le\Mass(R_r)$.  Together with
\eqref{eq:slice-derivative} and
\eqref{eq:competitor-mass}, this implies
\begin{equation}\label{eq:differential-squashing}
 F(r)\le c_s\delta F'(r)+c_sQ
\end{equation}
for almost every $r\in[1/2,1]$.  Setting $G=F-c_sQ$, one obtains
\[
 G'(r)-\frac{1}{c_s\delta}G(r)\ge0.
\]
Hence $r\mapsto e^{-r/(c_s\delta)}G(r)$ is nondecreasing.  Evaluating at
$r=1/2$ and $r=1$ proves \eqref{eq:squashing-integrated}, after increasing
$c_s$ harmlessly.
\end{proof}

\begin{remark}\label{rem:one-sided}
In addition to the projected-current identity
\eqref{eq:projection-local}, the proof uses from the closeness
assumption only the one-sided inclusion
\[
 \supp T\cap B_2(q)
 \subset \{\dist_{\Hyp}(\cdot,L)\le\delta\}.
\]
It does not use the reverse Hausdorff inclusion.  Consequently,
with \eqref{eq:projection-local} retained, the two-sided closeness
hypothesis in \cite[Lemma~3.2]{JiangXie} may be weakened, for the
present argument, to the one-sided tubular inclusion
\eqref{eq:tube-local}.
\end{remark}

\section{Proofs of the Main Theorems}\label{sec:proofs}

\subsection{Solution of Lin's Problem 1}

\begin{proof}[Proof of \cref{thm:Lin-problem}]
Set $L=\R^n\times\{0\}$ and let
$\delta=d_{\rm H}(D_1,\supp T)$.  The one-sided part of the Hausdorff bound
implies
\begin{equation}\label{eq:Lin-global-slab}
 \supp T\subset\{x:\dist(x,D_1)\le\delta\}
 \subset\{x:\dist(x,L)\le\delta\}.
\end{equation}
If $\delta=0$, then $T$ is supported in $L$ and the projection identity gives
$T=Q\cur{D_1}$, so the conclusion is immediate.  We henceforth assume
$\delta>0$.
Fix once and for all $R=1/16$.  If $a\in D_{1/2}$, then
$B_{4R}(a)\subset\mathcal C_1$.  Hence the varifold $V_T$ is stationary in
$B_{4R}(a)$, and \cref{cor:CM-local} gives
\begin{equation}\label{eq:Lin-local-doubling}
 \Mass(T\restr B_{2R}(a))
 \le C_{\rm CM}(n)\Mass(T\restr B_R(a)),
\end{equation}
provided $\delta/R\le\delta_{\rm CM}(n)$.

Put $\theta=\delta/R$.  We next localize the global projection identity.  If
$y\in D_{(2-2\theta)R,{\rm E}}^L(a)$ and $x\in\supp T$ satisfies
$\mathbf p(x)=y$, then \eqref{eq:Lin-global-slab} gives
$|x-y|\le\delta$, and therefore
\[
 |x-a|\le |y-a|+|x-y|<2R-\delta<2R.
\]
No point of $\supp T\setminus B_{2R}(a)$ can consequently contribute to the
pushforward over this disk.  Since
$D_{(2-2\theta)R,{\rm E}}^L(a)\subset D_1$, restriction of
$\mathbf p_\#T=Q\cur{D_1}$ gives
\begin{equation}\label{eq:Lin-local-projection}
 \mathbf p_\#(T\restr B_{2R}(a))
 \restr D_{(2-2\theta)R,{\rm E}}^L(a)
 =Q\cur{D_{(2-2\theta)R,{\rm E}}^L(a)}.
\end{equation}
Thus \cref{lem:Euclidean-squashing} applies and yields
\begin{equation}\label{eq:Lin-local-squashing}
 \Mass(T\restr B_R(a))
 \le e^{-1/(c_E\theta)}\Mass(T\restr B_{2R}(a))+c_EQR^n.
\end{equation}
Combining \eqref{eq:Lin-local-doubling} and
\eqref{eq:Lin-local-squashing}, and choosing $\delta_L(n)$ so that
\[
 \frac{\delta_L}{R}\le\min\{\delta_{\rm CM}/2,1/16\},
 \qquad
 C_{\rm CM}e^{-R/(c_E\delta_L)}\le\frac12,
\]
we obtain the uniform local estimate
\begin{equation}\label{eq:Lin-local-absolute}
 \Mass(T\restr B_R(a))\le C(n)QR^n
 \qquad(a\in D_{1/2}).
\end{equation}

Choose $a_1,\dots,a_N\in D_{1/2}$ so that the disks
$D_{R/2,{\rm E}}^L(a_j)$ cover $D_{1/2}$, with $N\le C(n)$.  Decrease
$\delta_L$ once more so that $\delta_L<R/2$.  If
$x\in\supp T\cap\mathcal C_{1/2}$, then
$\mathbf p(x)\in D_{R/2,{\rm E}}^L(a_j)$ for some $j$ and
$|x-\mathbf p(x)|\le\delta_L$ by \eqref{eq:Lin-global-slab}; hence
$x\in B_R(a_j)$.  Therefore
\[
 \Mass(T\restr\mathcal C_{1/2})
 \le\sum_{j=1}^N\Mass(T\restr B_R(a_j))
 \le C_L(n)Q.
\]
This proves \eqref{eq:Lin-conclusion}.  Notice that neither this covering
argument nor the two local estimates use indecomposability.
\end{proof}

\subsection{The hyperbolic fixed-scale estimate}

\begin{proof}[Proof of \cref{thm:hyperbolic-comparison}]
Choose $\eps_D(n)\le\eps_0(n)$ small enough for
\cref{lem:r-rbar,lem:target-annulus,lem:transverse-energy} and the cutoff
constructions above.  Then all preceding estimates apply to $V$.
Write $b=|P_S\nabla r|^2$.  Away from $p$,
\[
 1=|\nabla r|^2=b+a.
\]
By \cref{lem:tilt-comparison,lem:transverse-energy},
\begin{equation}\label{eq:a-integral}
 \int_{B_{5/4}\setminus B_{1/4}}a\,dV
 \le C(n)\mu_V(B_{1/4}).
\end{equation}
Combining \eqref{eq:a-integral} with \cref{lem:radial-energy} gives
\[
 \mu_V(B_{5/4}\setminus B_{1/4})
 =\int_{B_{5/4}\setminus B_{1/4}}(a+b)\,dV
 \le C(n)\mu_V(B_{1/4}).
\]
Adding $\mu_V(B_{1/4})$ and taking $C_D(n)$ to be the resulting constant
proves \eqref{eq:mass-comparison-main}.  To justify the
piecewise linear cutoff in \eqref{eq:psi}, first perturb its break radii so
that the corresponding spheres have zero $\mu_V$-mass, apply stationarity to
smooth approximations with uniformly bounded derivatives on the transition
annuli, and pass monotonically to the prescribed radii.  The Radon property
of $\mu_V$ and monotone convergence then give the result for the open balls
appearing in the statement.
\end{proof}

\subsection{Proof of the shifted-center form}

\begin{proof}[Proof of \cref{cor:shifted-comparison}]
Let $p_0$ be the intrinsic nearest point of $q$ on $L$.  After decreasing
$\eps_D$, assume $d_{\Hyp}(p_0,q)<1/8$.  Then
\[
 B_{3/2}(p_0)\subset B_{13/8}(q)\subset B_2(q),
\]
so \cref{thm:hyperbolic-comparison} applies at $p_0$.  Moreover,
\[
 B_1(q)\subset B_{9/8}(p_0)\subset B_{5/4}(p_0),
 \qquad
 B_{1/4}(p_0)\subset B_{3/8}(q)\subset B_{1/2}(q).
\]
Consequently,
\[
 \mu_V(B_1(q))
 \le\mu_V(B_{5/4}(p_0))
 \le C_D\mu_V(B_{1/4}(p_0))
 \le C_D\mu_V(B_{1/2}(q)).
\]
\end{proof}

\subsection{Proof of the absolute local mass bound}

\begin{proof}[Proof of \cref{thm:local-mass}]
Set
\[
 F(r)=\Mass_{\Hyp^{\mathfrak m}}\bigl(T\restr B_r(q)\bigr),
 \qquad r\in[1/2,1],
\]
and let $V_T$ be the integral varifold associated with $T$.  By
\eqref{eq:local-minimizer-definition}, \eqref{eq:no-boundary-local}, and the
first variation of mass, $V_T$ is stationary in $B_2(q)$.  After requiring
$\delta_*$ to be smaller than the threshold in
\cref{cor:shifted-comparison}, that corollary gives
\begin{equation}\label{eq:doubling-F}
 F(1)\le C_DF(1/2).
\end{equation}
By \cref{lem:squashing},
\[
 F(1)
 \le C_D\exp\!\left(-\frac{1}{2c_s\delta}\right)F(1)
      +C_Dc_sQ.
\]
Decrease $\delta_*(n)$ once more so that
\[
 C_D\exp\!\left(-\frac{1}{2c_s\delta_*}\right)\le\frac12.
\]
Absorption yields $F(1)\le2C_Dc_sQ$, which is
\eqref{eq:mass-local}.
\end{proof}

\subsection{Removal of the local mass-growth hypothesis}

\begin{proof}[Proof of \cref{thm:remove-growth}]
Apply \cref{prop:boundary-localization}.  For
$P\in\supp T\cap G_{R_0}$ with $t=x^{\mathfrak m}(P)<t_0$, its three
conclusions verify \eqref{eq:no-boundary-local}--\eqref{eq:projection-local}
for $\widehat T_P=(\Phi_P)_\#T$, $q=q_P$, $L=L_P$, $p=p_P$, $Q=1$, and
$\delta=C_\Gamma t^\alpha$.  Decreasing $t_0$ if necessary ensures
$\delta\le\delta_*(n)$.  Hence \cref{thm:local-mass} gives
\[
 \Mass_{\Hyp^{\mathfrak m}}
 \bigl(\widehat T_P\restr B_1(q_P)\bigr)\le C(n).
\]
Since \eqref{eq:boundary-rescaling} is a hyperbolic isometry, this is exactly
\eqref{eq:uniform-at-infinity}.

In \cite{JiangXie}, condition~(1.10) is used in
\cite[Lemma~3.4]{JiangXie} only to convert the exponentially small squashing
factor into a uniform unit-ball mass bound.  The preceding estimate supplies
that bound directly.  The remainder of the proof of
\cite[Theorem~1.4]{JiangXie} is unchanged.  The later PDE arguments begin
with the graphical $C^{1,\alpha}$ conclusion of that theorem and do not use
condition~(1.10) independently; therefore the conclusions of
\cite[Theorems~1.6, 1.7, 1.9, and~1.10]{JiangXie},
\cite[Corollary~1.8]{JiangXie}, and \cite[Theorem~7.4]{JiangXie} remain valid without condition~(1.10).  Hence the later boundary regularity,
expansion, and convergence results listed in the Introduction follow under
their respective additional regularity hypotheses.
\end{proof}

\section{Further applications of the mass-bound result}\label{sec:app}

In this section, we discuss further  applications of Theorem \ref{thm:Lin-problem}. 
As a corollary of Theorem \ref{thm:Lin-problem} and Allard’s interior regularity theorem, we first obtain the following height‑only interior regularity theorem in multiplicity one case. No a priori mass‑ratio bound is required.
\begin{theorem}
\label{thm:height-only}
Let $\alpha\in(0,1)$. There exist
\[
  \delta_{\mathrm{reg}}
  =\delta_{\mathrm{reg}}(n,k,\alpha)>0,
  \qquad
  C_{\mathrm{reg}}
  =C_{\mathrm{reg}}(n,k,\alpha)<\infty
\]
with the following property. Let $T$ be a locally area-minimizing integral
$n$-current in $\mathbb R^{n+k}$ satisfying
\[
  \mathbf p_\#T=\llbracket D_1\rrbracket,
  \qquad
  (\partial T)\llcorner\mathcal C_1=0,
  \qquad
  d_{\mathrm H}(D_1,\operatorname{supp}T)\leq\delta
  \leq\delta_{\mathrm{reg}}.
\]
Then there exists $u\in C^\infty(D_{1/4};\mathbb R^k)$ such that
\[
  T\llcorner\mathcal C_{1/4}=\mathbf G_u
\]
and
\[
  \|u\|_{C^{1,\alpha}(D_{1/4})}
  \leq C_{\mathrm{reg}}\delta.
\]
\end{theorem}
The proof is standard, and the theorem yields a small‑oscillation estimate for high‑codimensional area‑minimizing graphs.
\begin{corollary}
\label{cor:small-oscillation}
Let $\alpha\in(0,1)$. There exist
$\delta_{\mathrm g}=\delta_{\mathrm g}(n,k,\alpha)>0$
and $C=C(n,k,\alpha)<\infty$ such that the following holds.
Suppose $u:D_1\to\mathbb R^k$ is smooth and its graph current
$\mathbf G_u$ is area minimizing in $\mathcal C_1$. If
\[
  \operatorname{osc}_{D_1}u\leq\delta_{\mathrm g},
\]
then
\[
  \|Du\|_{C^{0,\alpha}(D_{1/4})}
  \leq C\operatorname{osc}_{D_1}u.
\]
\end{corollary}

There is also a global rigidity consequence: a degree‑one Bernstein‑type rigidity result under sublinear confinement.  
 Put
$L=\mathbb R^n\times\{0\}$ and let
$\mathbf p:\mathbb R^{n+k}\to L$ be orthogonal projection.  Define
\begin{equation}
  \mathsf h_T(R)
  =\sup\bigl\{\operatorname{dist}(x,L):
        x\in\operatorname{supp}T,
        |\mathbf p(x)|\leq R\bigr\},
  \label{eq:global-height-function}
\end{equation}
which may take the value $+\infty.$

\begin{theorem}
\label{thm:degree-one-bernstein}
Let $T$ be a locally area-minimizing integral $n$-current in
$\mathbb R^{n+k}$ such that
\begin{equation}
  \partial T=0,
  \qquad
  \mathbf p_\#T=\llbracket L\rrbracket
  \label{eq:bernstein-hypotheses}
\end{equation}
as locally finite currents, and
\begin{equation}
  \lim_{R\to\infty}\frac{\mathsf h_T(R)}{R}=0,
  \label{eq:bernstein-sublinear-height}
\end{equation}
then there is a vector $b\in L^\perp$ such that
\begin{equation}
  T=\llbracket L+b\rrbracket.
  \label{eq:bernstein-conclusion}
\end{equation}
\end{theorem}

\begin{proof}
For $R>0$, let $\eta_R(x)=R^{-1}x$ and set
\[
  T_R=(\eta_R)_\#T.
\]
Fix $A>1$.  On the portion of $\operatorname{supp}T_R$ whose projection
lies in $D_{2A}^L(0)$, one has
\begin{equation}
  \operatorname{dist}(x,L)
  \leq\varepsilon_R(A)
  :=\frac{\mathsf h_T(2AR)}{R},
  \qquad
  \varepsilon_R(A)\longrightarrow0.
  \label{eq:blow-down-height}
\end{equation}
The reverse closeness to $L$ also holds.  Indeed, the height assumption \eqref{eq:bernstein-sublinear-height} implies that
$\mathbf p|_{\operatorname{supp}T}$ is proper and
$\mathbf p(\operatorname{supp}T)$ is closed.  Since
\[
  \operatorname{supp}(\mathbf p_\#T)=L,
\]
it follows that $\mathbf p(\operatorname{supp}T)=L$.  Thus every point of
$D_{2A}^L(0)$ has a preimage in $\operatorname{supp}T_R$, and that preimage
has height at most $\varepsilon_R(A)$.

Furthermore,
\[
  \partial T_R=0,
  \qquad
  \mathbf p_\#T_R=\llbracket L\rrbracket.
\]
One has
\begin{equation}
  T_R\rightharpoonup\llbracket L\rrbracket,
  \qquad
  V_{T_R}\rightharpoonup\mathbf v(L,1)
  \quad\text{locally as }R\to\infty,
  \label{eq:blow-down-convergence}
\end{equation}
where $\mathbf v(L,1)$ is the varifold associated with $L.$

Fix $q\in\operatorname{supp}T$.  Since $q/R\to0$ and
$\mathcal H^n(L\cap\partial B_1(0))=0$,
\eqref{eq:blow-down-convergence} implies
\begin{equation}
  \lim_{R\to\infty} \Theta_T(q,R):= \lim_{R\to\infty}
  \frac{\|T\|(B_R(q))}{\omega_nR^n}=1.
  \label{eq:density-at-infinity}
\end{equation}
The monotonicity formula says that
\[
  r\longmapsto\Theta_T(q,r)
\]
is nondecreasing.  A nonzero area-minimizing integral current has density
at least one at every point of its support.  Hence
\[
  1\leq\lim_{r\downarrow0}\Theta_T(q,r)
  \leq\Theta_T(q,r)
  \leq\lim_{R\to\infty}\Theta_T(q,R)=1.
\]
Consequently,
\begin{equation}
  \Theta_T(q,r)=1
  \qquad\text{for every }r>0.
  \label{eq:all-mass-ratios-one}
\end{equation}
Equality in the monotonicity formula makes $T$ a cone about $q$; a
stationary integral cone of density one is a multiplicity-one plane.
Therefore
\[
  T=\llbracket P\rrbracket
\]
for an oriented affine $n$-plane $P$.

The projected-current identity implies that $\mathbf p|_P$ has rank $n$
and positive degree one, so $P$ is a graph over $L$:
\[
  P=\{x+Ax+b:x\in L\}
\]
for a linear map $A:L\to L^\perp$ and some $b\in L^\perp$.  If $A\neq0$,
then $\mathsf h_T(R)/R$ stays bounded below by a positive constant along a
sequence $R\to\infty$.  The sublinear-confinement hypothesis forces
$A=0$, and hence $P=L+b$.
\end{proof}

\end{document}